\documentclass[a4paper, 11pt]{amsart}
\usepackage{amsmath,amssymb,mathrsfs,amsthm}
\usepackage{array,enumerate}
\usepackage{latexsym}
\usepackage{mathrsfs}
\usepackage{amscd}
\usepackage[all]{xy}
\usepackage{graphicx}
\usepackage[svgnames]{xcolor}
\usepackage{tikz}
\usetikzlibrary{intersections}
\theoremstyle{definition}
\newtheorem{thm}{Theorem}[section]
\newtheorem{dfn}[thm]{Definition}
\newtheorem{note-dfn}[thm]{Notation-Definition}

\newtheorem{prop}[thm]{Proposition}
\newtheorem{cor}[thm]{Corollary}

\newtheorem{lem}[thm]{Lemma}
\newtheorem{rem}[thm]{Remark}

\newtheorem{prop-dfn}[thm]{Proposition-Definition}

\newcommand{\Z}{\mathbb{Z}}
\newcommand{\Q}{\mathbb{Q}}

\newcommand{\F}{\mathbb{F}}
\newcommand{\Spec}{\mathrm{Spec}\,}
\newcommand{\deq}{\overset{\mathrm{def}}{=}}

\title[Shafarevich conjecture for proper hyperbolic polycurves]{
The Shafarevich conjecture and some extension theorems for proper hyperbolic polycurves}
\author{Ippei Nagamachi,\ Teppei Takamatsu}
\date{\today}
\email{nagachi@ms.u-tokyo.ac.jp,\ teppei@ms.u-tokyo.ac.jp}

\begin{document}

\maketitle

\begin{abstract}
In this paper, we prove the Shafarevich conjecture for proper hyperbolic polycurves, which is a higher dimensional analogue of that for proper hyperbolic curves.
First, we study theories of proper hyperbolic polycurves over regular schemes.
For example, we generalize the moduli theory of Kodaira fibrations due to Jost and Yau  \cite{JY}.
We also show the N\'eron property of proper smooth models of proper hyperbolic polycurves over Dedekind schemes under an assumption on residual characteristics.
We then apply these extension theories to the proof of the Shafarevich conjecture for proper hyperbolic polycurves.
\end{abstract}

\setcounter{section}{-1}

\section{Introduction}
The Shafarevich conjecture for proper hyperbolic curves, which was proved by Faltings, states the finiteness of isomorphism classes of proper hyperbolic curves of fixed genus over a fixed number field admitting good reduction away from a fixed finite set of finite places.
In this paper, we shall establish the generalization of this theorem for the class of proper hyperbolic polycurves,
that is, 
varieties $X$ which admit a structure of successive smooth fibrations (called a sequence of parameterizing morphisms (cf.\ Definition \ref{hyperbolicpoly})) 

\begin{equation*}
(\mathcal{S}): X = X_n \rightarrow X_{n-1} \rightarrow \cdots 
\rightarrow X_1 
\rightarrow X_{0}= \mathrm{Spec}\,K 
\end{equation*}
whose fibers are proper hyperbolic curves.
Let $F$ be a number field, $\mathfrak{p}$ a finite place of $F$, and $O_{F,\,\mathfrak{p}}$ the valuation ring of $F$ at $\mathfrak{p}$.
For any proper hyperbolic polycurve $X$ over $F$, 
we say that $X$ has good reduction at $\mathfrak{p}$ if $X$ admits a smooth proper model over $\mathcal{O}_{F,\,\mathfrak{p}}$ (cf.\ Definition \ref{reddef}).
Our main theorem is the following:

\begin{thm}[see Theorem \ref{Shaf} for a more general statement]
Let $F$ be a number field, and $S$ a finite set of finite places of $F$. 
Let $\chi$ be an integer, and $n$ a positive integer.
Then there exist at most finitely many isomorphism classes of proper hyperbolic polycurves of dimension $n$ with Euler-Poincar\'e characteristic $\chi$ over $F$ which has good reduction outside $S$.
\label{introShaf}
\end{thm}
As a corollary of Theorem \ref{introShaf}, we give another proof of Sawada's finiteness theorem of isomorphism classes of proper hyperbolic polycurves with prescribed fundamental groups (cf.\,Corollary \ref{sawcor} and Remark \ref{Sawada}).

The main tool of the proof of Theorem \ref{introShaf} is the Shafarevich conjecture for proper hyperbolic curves over finitely generated fields of characteristic $0$, which was also proved by Faltings.
To use Faltings's result inductively, we study the structure of integral models of proper hyperbolic polycurves.
Precisely, we shall give the following result:

\begin{thm}[cf.\,Theorem \ref{uniqueness}]
Let $T$ be a connected Noetherian regular scheme, $K(T)$ the field of fractions of $T$, and $\mathfrak{X} \to T$ a proper smooth scheme.
Write $X$ for the scheme $\mathfrak{X}\times_{T}\Spec K(T)$.
Suppose that $X$ is a proper hyperbolic polycurve over $K(T)$.
Moreover, suppose that the residual characteristic of every point of $T$ of codimension $1$ is sufficiently large (see Theorem \ref{uniqueness} for the precise bound) or equal to $0$.
Then, for any sequence of parameterizing morphisms $(\mathcal{S})$ of $X\to \Spec K(T)$, there exists a unique sequence of parameterizing morphisms $(\mathfrak{S}')$ of a proper hyperbolic polycurve $\mathfrak{X}'\to T$ (up to canonical isomorphism) such that the base change of $(\mathfrak{X}', (\mathfrak{S}'))$ to $\Spec K(T)$ is isomorphic to $(X,(\mathcal{S}))$ (cf.\,Definition \ref{hyperbolicpoly}.2) and $\mathfrak{X}$ is canonically isomorphic to $\mathfrak{X}'$.
\label{introstr}
\end{thm}
Theorem \ref{introstr} is a sort of generalization of the result of Jost and Yau \cite{JY}.
(In \cite{JY}, proper hyperbolic polycurves of relative dimension $2$ over complex manifolds are treated.)
We prove Theorem \ref{introstr} by applying the results on N\'eron models of hyperbolic curves by Liu-Tong \cite{Liu} and the purity of proper hyperbolic polycurves over regular schemes (cf.\,Theorem \ref{polypurity}).

The content of each section is as follows: 
In Section \ref{onreg}, we give the precise definition of a proper hyperbolic polycurve and the proof of the purity of proper hyperbolic polycurves over regular schemes.
In Section \ref{neronsec}, we discuss structures of smooth models of a proper hyperbolic polycurve over a Dedekind scheme.
In Section \ref{sectshaf}, we give the proof of the Shafarevich conjecture for proper hyperbolic polycurves by using the results of Section \ref{neronsec} and Faltings's result.
In Section \ref{sectpresc}, we give another proof of the Sawada's finiteness theorem for proper polycurves in the case where their coefficient fields are finitely generated over $\Q$.

\subsection*{acknowledgement}
The first author was supported by Iwanami Fujukai Foundation.
The second author would like to thank Qing Liu for helpful comments on Proposition \ref{neronproperty}.
The second author is supported by the FMSP program at the University of Tokyo.
This work was supported by the Research Institute for Mathematical Sciences, a Joint Usage/Research Center located in Kyoto University.

\section{Proper hyperbolic polycurves over regular schemes}
In this section, we discuss properties of proper hyperbolic polycurves over regular schemes (cf.\,Theorem \ref{polypurity}).
We start with the definition of a proper hyperbolic polycurve.

\begin{dfn}
Let $S$ be a scheme and $X$ a scheme over $S$.
\begin{enumerate}
\item We shall say that $X$ is a {\it proper hyperbolic curve} over $S$ if
the structure morphism $X \rightarrow S$ is proper, smooth, and of relative dimension $1$ with geometrically connected fibers of genus $g \geq 2$.
\item We shall say that $X$ is a {\it proper hyperbolic polycurve} (of relative dimension $n$) over $S$ if
there exists a (not necessarily unique) factorization
\begin{equation*}
(\mathcal{S}) : X = X_{n} \rightarrow X_{n-1} \rightarrow \ldots \rightarrow X_{1} \rightarrow X_{0} = S
\end{equation*}
of the structure morphism $X \rightarrow S$ such that, for each $i \in \{ 1, \ldots ,n \}$, $X_{i} \rightarrow X_{i-1}$ is a proper hyperbolic curve.
We refer to the above factorization of the morphism $X \rightarrow S$ as a {\it sequence of parametrizing morphisms}.
Let $g_{i}$ be the genus of the curve $X_{i} \to X_{i-1}$ for each $1 \leq i \leq n$.
We write $g_{\mathcal{S}} \deq \underset{1 \leq i \leq n}{\mathrm{max}} g_{i}$.
We also write $g_{X} \deq \underset{\mathcal{S}}{\mathrm{min}}\, g_{\mathcal{S}}$, where $\mathcal{S}$ ranges over the sequences of parametrizing morphisms of $X \to S$.
In the case where we consider a pair of a proper hyperbolic polycurve $X\to S$ and a sequence of parametrizing morphisms $(\mathcal{S})$ of $X \to S$, we write $(X, (\mathcal{S}))$.
We refer to such a pair as a {\it proper hyperbolic polycurve with a sequence of parametrizing morphisms}.
We shall say that two proper hyperbolic polycurves (over $S$) with a sequence of parametrizing morphisms $(X, (\mathcal{S}))$ and $(X', (\mathcal{S'}))$ are isomorphic if there exists an $S$-isomorphism between proper hyperbolic polycurves of relative dimension $i$ over $S$ defined by $(\mathcal{S})$ and $(\mathcal{S}')$ for each $1 \leq i \leq n$ such that these isomorphisms are compatible with the sequence of parametrizing morphisms $(\mathcal{S})$ and $(\mathcal{S}')$.

\item Let $X$ be a proper hyperbolic polycurve of relative dimension $n$ over $S$.
Let
\begin{equation*}
X = X_{n} \rightarrow X_{n-1} \rightarrow \ldots \rightarrow X_{1} \rightarrow X_{0} = S
\end{equation*}
be a sequence of parametrizing morphisms of $X \rightarrow S$.
Write $g_{i}$ for the genus of the proper hyperbolic curve $X_{i} \to X_{i-1}$ for each $1\leq i \leq n$.
We refer to the nonzero integer
$$\chi(X)=\underset{1\leq i \leq n}{\prod}(2-2g_{i})$$
as {\it the Euler-Poincar\'e characteristic} of the proper hyperbolic polycurve $X \to S$.
Note that $\chi(X)$ does not depend on the choice of a sequence of parametrizing morphisms of the proper hyperbolic polycurve $X \rightarrow S$ by Lemma \ref{chi}.
It holds that $2^{n}$ divides $\chi(X)$.
Note that we cannot determine $n$ from $\chi(X)$.
\end{enumerate}
\label{hyperbolicpoly}
\end{dfn}
\begin{lem}
Let $X$ be a proper hyperbolic polycurve of relative dimension $n$ over $S$.
$\chi(X)$ does not depend on the choice of a sequence of parametrizing morphisms of $X \rightarrow S$.
\label{chi}
\end{lem}
\begin{proof}
Let
\begin{equation*}
X = X_{n} \rightarrow X_{n-1} \rightarrow \ldots \rightarrow X_{1} \rightarrow X_{0} = S
\end{equation*}
be a sequence of parametrizing morphisms of $X \rightarrow S$.
Write $g_{i}$ for the genus of the proper hyperbolic curve $X_{i} \to X_{i-1}$ for each $1\leq i \leq n$.
We may assume that $S$ is the spectrum of an algebraically closed field $k$.
Let $l\neq p$ be a prime number and $\chi(X, \F_{l})$ the Euler characteristic of the trivial \'etale sheaf $\F_{l}$ on $X$.
It suffices to show that $\underset{1\leq i \leq n}{\prod}(2-2g_{i})=\chi(X, \F_{l})$.
By induction on $n$ and the Leray spectral sequence for $X_{n}\to X_{n-1}$, one can verify this by using \cite[COROLLAIRE 2.11]{Ill}.
\end{proof}
\begin{rem}
In the case where $S$ has a point $s$ whose residual characteristic is $0$, Lemma \ref{chi} follows immediately from the fact that the Euler-Poincar\'e characteristic of a proper hyperbolic polycurve $X\to S$ can be determined by the \'etale fundamental groups of the scheme $X\times_{S}\overline{s}$.
Here, $\overline{s}$ is a geometric point of $X$ over $s$.
Indeed, let $l$ be a prime number.
Write $\Delta$ for the \'etale fundamental group of the scheme $X\times_{S}\overline{s}$, $H^{i}(\Delta, \F_{l})$ for the $i$-th cohomology group of the trivial $\Delta$-module $\F_{l}$, and $\chi(\Delta, \F_{l})$ for the Euler characteristic 
$\underset{0\leq i \leq \infty}{\sum}
(-1)^{i}\mathrm{dim}_{\F_{l}}H^{i}(\Delta, \F_{l})$.
Then one can verify that $\chi(X)=\chi(\Delta, \F_{l})$.
\label{fundEuler}
\end{rem}
In this paper, we shall say that a scheme $T$ is {\it Dedekind} if $T$ is a $1$-dimensional connected Noetherian normal separated scheme.
\begin{dfn}
Let $T$ be a Dedekind scheme, $K(T)$ the function field of $T$, $\eta$ the generic point of $T$, and $X \to \mathrm{Spec}\,K(T)$ a proper smooth morphism with geometrically connected fibers.
\begin{enumerate}
\item We shall say that $X$ has good reduction if there exists a proper smooth $T$-scheme $\mathfrak{X}$ 
whose generic fiber $\mathfrak{X} _{\eta}$ is isomorphic to $X$ over $K(T)$. 
We refer to such $\mathfrak{X}$ as a {\it smooth model} of $X$.
\item Suppose that $X \to \Spec K(T)$ is a proper hyperbolic polycurve.
Let
\begin{equation*}
(\mathcal{S}) : X = X_{n} \rightarrow X_{n-1} \rightarrow \ldots \rightarrow X_{1} \rightarrow X_{0} = \Spec K(T)
\label{defseq}
\end{equation*}
be a sequence of parametrizing morphism of $X \to \Spec K(T)$.
We shall say that $X$ {\it has good reduction with respect to} $(\mathcal{S})$ if there exist a proper hyperbolic polycurve $\mathfrak{X} \to T$ and a sequence of parametrizing morphisms
$$(\mathcal{S}') : \mathfrak{X} = \mathfrak{X}_{n} \rightarrow \mathfrak{X}_{n-1} \rightarrow \ldots \rightarrow \mathfrak{X}_{1} \to \mathfrak{X}_{0} = T$$
of $\mathfrak{X} \to T$ such that the proper hyperbolic polycurve with a sequence of parametrizing morphisms defined by the base change of the sequence $(\mathcal{S}')$ to $\Spec K(T)$ is isomorphic to $(X, (\mathcal{S}))$.
\item Let $\mathfrak{X}$ be a separated, smooth, and of finite type scheme over $T$ whose generic fiber is isomorphic to $X$ over $K(T)$.
We shall say that $\mathfrak{X}$ is the N\'eron model of $X$ if the following property, called {\it N\'eron mapping property}, is satisfied:\\
for any smooth scheme $\mathfrak{Y}$ over $T$, the canonical map
$$\mathrm{Mor}_{T}(\mathfrak{Y} , \mathfrak{X}) \to \mathrm{Mor}_{K(T)}(\mathfrak{Y}\times_{T} \Spec K(T), X)$$
is a bijection.
Here, $\mathrm{Mor}_{T}(\mathfrak{Y} , \mathfrak{X})$ is the set of morphisms from $\mathfrak{Y}$ to $\mathfrak{X}$ over $T$, and $\mathrm{Mor}_{K(T)}(\mathfrak{Y}\times_{T} \Spec K(T), X)$ is the set of morphisms from $\mathfrak{Y}\times_{T} \Spec K(T)$ to $X$ over $K(T)$.
\end{enumerate}
\label{reddef}
\end{dfn}

\begin{rem}
If a smooth model of a proper hyperbolic curve exists, it is unique up to canonical isomorphism (cf.\,\cite{DM}).
This also follows from Lemma \ref{moduliunique} or Proposition \ref{neronproperty}.2.
\label{hyp-uniq}
\end{rem}

\begin{lem}
Let $S$ be an irreducible normal scheme and $K(S)$ the function field of $S$.
Let $C_{1}$ and $C_{2}$ be proper hyperbolic curves over $S$, and $\phi$ an isomorphism $C_{1}\times_{S}\Spec K(S) \cong C_{2}\times_{S}\Spec K(S)$ over $K(S)$.
Then there exists a unique isomorphism $\Phi: C_{1} \cong C_{2}$ over $S$ whose base change to $\Spec K(S)$ coincides with $\phi$.
\label{moduliunique}
\end{lem}

\begin{proof}
Lemma \ref{moduliunique} follows from the argument given in the discussion entitled 
``Curves''
in \cite[\S 0]{Moch}.
For the convenience of the reader, we give the proof here.
We may assume that a prime $l$ is invertible on $S$ (which, by Zariski localization, we may assume without loss of generality).
Let $g$ be the genus of $C_{i}$.
Then the moduli stack $\mathcal{M}_{g}$ of proper smooth curve of genus $g\,(\geq 2)$ over $\Spec \Z[1/l]$ has a finite \'etale covering from a scheme.

Let $C \to S$ be a proper hyperbolic curve of genus $g$, $c: S\to \mathcal{M}_{g}$ the $1$-morphism defined by $C$, and $\iota: \Spec K(S) \to S$ the natural morphism.
Let $M \to \mathcal{M}_{g}$ be a finite \'etale covering from a scheme, $\Spec L$ the scheme representing $\Spec K(S) \times_{\mathcal{M}_{g}}M$.
Write $\phi_{L}$ for the base change of $\phi$ to $\Spec L$.
Let $S'$ be the normalization of $S$ in $\Spec L$.
Then the scheme $S'$ represents $S\times_{\mathcal{M}_{g}}M$.
Since $\mathcal{M}_{g}$ is separated over $\Spec \Z[1/l]$, there exists a unique isomorphism
$$\Phi_{S'}: C_{1}\times_{S}S' \cong C_{2}\times_{S}S'$$
whose base change to $\Spec L$ coincides with $\phi_{L}$.
Hence, the desired morphism $\Phi$ uniquely exists.
\end{proof}

\begin{prop}[cf.\,\cite{Liu}]
Let $T$, $K(T)$, and $X$ be as in Definition \ref{reddef}.
\begin{enumerate}
\item Let $\mathfrak{X}$ be a smooth model of $X$.
Suppose that each closed fiber of the morphism $\mathfrak{X}\to T$ contains no rational curves.
Then $\mathfrak{X}$ is the N\'eron model of $X$ (\cite[Proposition 4.13]{Liu}).
\item Suppose that $X$ is a proper hyperbolic curve which has good reduction.
Then a smooth model of $X$ (cf.\,Remark \ref{hyp-uniq}) is the N\'eron model of $X$.
\item Let $Y \to \Spec K(T)$ be a proper smooth morphism with geometrically connected fibers.
Suppose that $X$ is a proper hyperbolic curve and that $Y$ has good reduction.
Moreover, suppose that there exists a $K(T)$-morphism from $Y$ to $X$.
Then $X$ has good reduction (\cite[Corollary 4.7]{Liu}).
\item If $X$ has a proper N\'eron model, any smooth model of $X$ is canonically isomorphic to the N\'eron model.
\end{enumerate}
\label{neronproperty}
\end{prop}

\begin{proof}
We only show assertion 2 and 4.
Assertion 2 follows from \cite[Theorem 1.1]{Liu} and the fact that a smooth model of $X$ is the minimal regular model of $X$.
Assertion 4 follows from van der Waerden's purity theorem (see \cite[Corollaire (21.12.16)]{EGA} for a more general statement).
\end{proof}

Note that Proposition \ref{neronproperty}.2 follows from Proposition \ref{neronproperty}.1.
Also, one can show Proposition \ref{neronproperty}.2 by the N\'eron mapping property of the N\'eron model of the Jacobian variety of $X$ (after replacing $T$ by the strict henselization of each closed point of $T$).

\begin{prop}
Let $T$ and $K(T)$ be as in Definition \ref{reddef}.
Let
$$\mathfrak{X}=\mathfrak{X}_{n} \to \mathfrak{X}_{n-1} \to \ldots \to \mathfrak{X}_{0} = T$$
be a proper hyperbolic polycurve.
Then $\mathfrak{X}$ is the N\'eron model of the scheme $\mathfrak{X} \times_{T} \Spec K(T)$.
\label{polyneron}
\end{prop}

\begin{proof}
By Proposition \ref{neronproperty}.1, it suffices to show that there exist no rational curves contained in the special fiber of $\mathfrak{X}$.
Since any morphism from a rational curve to a hyperbolic curve over a field is constant, a proper hyperbolic polycurve over a field contains no rational curves.
Hence, Proposition \ref{polyneron} holds.
\end{proof}

\begin{prop}[cf.\,\cite{Mor} and {\cite[Section 7]{Nag}}]
Let $Z$ be a connected Noetherian regular scheme.
\begin{enumerate}
\setlength{\itemindent}{-14pt}
\item
Let $K(Z)$ be the function field of $Z$ and $C_{K(Z)} \to \Spec K(Z)$ a proper hyperbolic curve.
The following are equivalent:
\begin{itemize}
\setlength{\itemindent}{-10pt}
\item There exists a proper hyperbolic curve $C_{Z} \to Z$ such that $C_{Z}\times_{Z}\Spec K(Z)$ is isomorphic to $C_{K(Z)}$ over $K(Z)$.
\item There exist a nonempty open subset $U$ of $Z$ satisfying that $Z\setminus U$ is of codimension $\geq 2$ in $Z$ and a proper hyperbolic curve $C_{U} \to U$ such that $C_{U}\times_{U}\Spec K(Z)$ is isomorphic to $C_{K(Z)}$ over $K(Z)$.
\item For any point $z\in Z$ of codimension $1$, there exists a proper hyperbolic curve $C_{O_{Z,z}} \to \Spec O_{Z,z}$ such that $C_{O_{Z,z}}\times_{O_{Z,z}}\Spec K(Z)$ is isomorphic to $C_{K(Z)}$ over $K(Z)$.
\end{itemize}
In this case, the scheme $C_{Z}$ (respectively, $C_{U}$; $C_{O_{Z,z}}$) is unique up to a canonical isomorphism over $Z$ (respectively, $U$; $O_{Z,z}$ for each point $z \in Z$ of codimension $1$).
Hence, the scheme $C_{Z}\times_{Z}U$ (respectively, $C_{Z}\times_{Z}\Spec O_{Z,z}$) is isomorphic to $C_{U}$ (respectively, $C_{O_{Z,z}}$) over $U$ (respectively, $O_{Z,z}$ for each point $z \in Z$ of codimension $1$) (cf.\,Remark \ref{globalization}).
\item Let $Y$ be a connected Noetherian regular scheme over $Z$, $V$ a nonempty open subset of $Y$ satisfying that $Y\setminus V$ is of codimension $\geq 2$ in $Y$, and $C' \to Z$ a proper hyperbolic curve. 
Then the restriction map
$$\mathrm{Mor}_{Z}(Y,C') \to \mathrm{Mor}_{Z}(V,C')$$
is bijective.
Here, $\mathrm{Mor}_{Z}(Y,C')$ (respectively, $\mathrm{Mor}_{Z}(V,C')$) is the set of morphisms from $Y$ to $C'$ over $Z$ (respectively, from $V$ to $C'$ over $Z$).
\end{enumerate}
\label{MB}
\end{prop}

\begin{proof}
Assertion 1 follows from \cite{Mor} and Lemma \ref{moduliunique}.
To show assertion 2, we may assume that $Y=Z$.
Then the assertion follows from \cite[Lemme 1]{Mor} (or \cite[Section 7]{Nag}).
\end{proof}

\begin{rem}
The latter part of Proposition \ref{MB}.1 does not holds in general.
Let $\mathbb{P}^{1}_{\Z}$ be the projective line over $\Spec \Z$.
Write $B_{p}$ for the scheme obtained by blowing up of $\mathbb{P}^{1}_{\Z_{(p)}}$ at some closed point of $\mathbb{P}^{1}_{\Z_{(p)}}$.
Then write $C_{p}$ for the scheme obtained by contraction of the strict transform of the special fiber of $\mathbb{P}^{1}_{\Z_{(p)}}$ in $B_{p}$.
Consider the family of smooth models $\{ D_{p} \to \Spec \Z_{(p)} \mid p\text{ is a prime number} \}$ of $\mathbb{P}^{1}_{\Q} \to \Spec \Q$, where $D_{p}$ is $\mathbb{P}^{1}_{\Z_{(p)}}$ or $C_{p}$.
One can verify that there exists a proper smooth curve $C \to \Spec \Z$ whose base change to $\Spec \Z_{(p)}$ is isomorphic to $C_{p}$ for all $p$ if and only if $D_{p} = \mathbb{P}^{1}_{\Z_{(p)}}$ for all but finite $p$.
\label{globalization}
\end{rem}

\begin{thm}
Let $Z$ be a connected Noetherian regular scheme.
\begin{enumerate}
\setlength{\itemindent}{-14pt}
\item
Let $K(Z)$ be the function field of $Z$ and
$$(\mathcal{S}_{K(Z)}): X_{n, K(Z)} \to \ldots \to X_{1, K(Z)}\to \Spec K(Z)$$ a sequence of parametrizing morphisms of a proper hyperbolic polycurve.
The following are equivalent:
\begin{itemize}
\setlength{\itemindent}{-10pt}
\item There exists a sequence of parametrizing morphisms of a proper hyperbolic polycurve
$$(\mathcal{S}_{Z}): X_{n, Z} \to \ldots \to X_{1, Z}\to Z$$
such that the base change of $(X_{n, Z}, (\mathcal{S}_{Z}))$ to $\Spec K(Z)$ is isomorphic to $(X_{n, K(Z)},(\mathcal{S}_{K(Z)}))$.
\item There exist a nonempty open subset $U$ of $Z$ satisfying that $Z\setminus U$ is of codimension $\geq 2$ in $Z$ and a sequence of parametrizing morphisms of a proper hyperbolic polycurve
$$(\mathcal{S}_{U}): X_{n, U} \to \ldots \to X_{1, U}\to U$$
such that the base change of $(X_{n, U}, (\mathcal{S}_{U}))$ to $\Spec K(Z)$ is isomorphic to $(X_{n,K(Z)},(\mathcal{S}_{K(Z)}))$.
\item For any point $z\in Z$ of codimension $1$, there exists a sequence of parametrizing morphisms of a proper hyperbolic polycurve
$$(\mathcal{S}_{O_{Z,z}}): X_{n, O_{Z,z}} \to \ldots \to X_{1, O_{Z,z}}\to  \Spec O_{Z,z}$$
such that the base change of $(X_{n, O_{Z,z}}, (\mathcal{S}_{O_{Z,z}}))$ to $\Spec K(Z)$ is isomorphic to $(X_{n,K(Z)},(\mathcal{S}_{K(Z)}))$.
\end{itemize}
In this case, $(X_{n, Z}, (\mathcal{S}_{Z}))$ (respectively, $(X_{n, U}, (\mathcal{S}_{U}))$; $(X_{n, O_{Z,z}}, (\mathcal{S}_{O_{Z,z}}))$) is unique up to a canonical isomorphism over $Z$ (respectively, $U$; $O_{Z,z}$ for each point $z \in Z$ of codimension $1$).
\item
Let $Y$ be a connected Noetherian regular scheme over $Z$, $V$ a nonempty open subset of $Y$ satisfying that $Y\setminus V$ is of codimension $\geq 2$ in $Y$, and $X'_{Z} \to Z$ a proper hyperbolic polycurve.
Then the restriction map
$$\mathrm{Mor}_{Z}(Y,X'_{Z}) \to \mathrm{Mor}_{Z}(V,X'_{Z})$$
is bijective.
Here, $\mathrm{Mor}_{Z}(Y,X'_{Z})$ (respectively, $\mathrm{Mor}_{Z}(V,X'_{Z})$) is the set of morphisms from $Y$ to $X'_{Z}$ over $Z$ (respectively, from $V$ to $X'_{Z}$ over $Z$).
\item Let $K(Z)$ and $(\mathcal{S}_{K(Z)})$ be as in assertion 1.
Suppose that the equivalent conditions of assertion 1 are satisfied.
Let $X_{Z}\to Z$ be a proper smooth morphism such that $X_{Z}\times_{Z}\Spec K(Z)$ is isomorphic to $X_{n,K(Z)}$ over $K(Z)$.
Then $X_{Z}$ is canonically isomorphic to $X_{n, Z}$ over $Z$.
\end{enumerate}
\label{polypurity}
\end{thm}

\begin{proof}
To show assertion 1 and 2, we may assume that $n=1$, in which case the assertions follow from Proposition \ref{MB}.1 and 2.
Next, we show assertion 3.
By Proposition \ref{neronproperty}.4, $X_{Z}\times_{Z}\Spec O_{Z,z}$ is canonically isomorphic to $X_{n,O_{Z,z}}$ over $\Spec O_{Z,z}$ for any point $z\in Z$ of codimension $1$.
Therefore, there exists an open subset $U$ of $Z$ such that $Z\setminus U$ is of codimension $\geq 2$ in $Z$ and $X_{Z}\times_{Z}U$ is canonically isomorphic to $X_{n,U}$ over $U$.
By assertion 2, there exists a canonical birational morphism $\phi : X_{Z}\to X_{n,Z}$ over $Z$.
$\phi$ is isomorphism by van der Waerden's purity theorem (cf.\,\cite[Corollaire (21.12.16)]{EGA}).
\end{proof}

\label{onreg}

\section{Existence of a smooth model of a proper hyperbolic polycurve with respect to a given sequence of parameterizing morphisms}
In this section, we discuss structures of smooth models of a proper hyperbolic polycurve over a Dedekind scheme.
For proper hyperbolic polycurves of relative dimension $2$ over complex manifolds, some of the main results of this section (part of Theorem \ref{uniqueness} and Corollary \ref{cor}) are proven in \cite{JY}.

\begin{dfn}[cf.\,{\cite[Theorem 1.2.3 and Theorem 1.3]{Nag}}]
Let $z$ be a positive integer.
Define a function $f_{z}(m)$ for $m \geq 2$ in the following way:
\begin{itemize}
\item
For $m=2$, $f_{z}(2) = z+1$.
\item
For $m=3$, $f_{z}(3) = 2^{z^{2}}$.
\item
For $m \geq 3$,
$$f_{z}(m+1)= (f_{z}(m)) \times (2^{z^{2} \times f_{z}(m)^{2}})^{f_{z}(m)}.$$
\end{itemize}
\end{dfn}

\begin{thm}
Let $Z$ be a connected Noetherian regular scheme, $K(Z)$ the field of fractions of $Z$, and $\mathfrak{X} \to Z$ a proper smooth scheme.
Write $X$ for the scheme $\mathfrak{X}\times_{Z}\Spec K(Z)$.
Suppose that $X$ is a proper hyperbolic polycurve of relative dimension $n$ over $K(Z)$.
\begin{enumerate}
\item
Let
$$(\mathcal{S}): X=X_{n} \to \ldots \to X_{1}\to \Spec K(Z)$$
be a sequence of parametrizing morphisms of a proper hyperbolic curve $X\to\Spec K(Z)$.
Suppose that the residual characteristic of every point of $Z$ of codimension $1$ is more than $2^{2g_{\mathcal{S}} \times f_{2g_{\mathcal{S}}}(n)}$ or equal to $0$.
Then there exists a unique sequence of parametrizing morphisms
$$(\mathfrak{S}): \mathfrak{X}_{n} \to \ldots \to \mathfrak{X}_{1}\to Z$$
(up to canonical isomorphism) such that the base change of $(\mathfrak{X}_{n},(\mathfrak{S}))$ to $\Spec K(Z)$ is isomorphic to $(X_{n},(\mathcal{S}))$ and $\mathfrak{X}$ is canonically isomorphic to $\mathfrak{X}_{n}$.
In particular, if the residual characteristic of every point of $Z$ of codimension $1$ is more than $2^{2g_{\mathcal{S}} \times f_{2g_{\mathcal{S}}}(n)}$ or equal to $0$, $\mathfrak{X}\to Z$ has a structure of a proper hyperbolic polycurve.
If, moreover, $Z$ is Dedekind, $\mathfrak{X}$ is the N\'eron model of $X$ over $Z$.
\item
Suppose that the residual characteristic of every point of $Z$ of codimension $1$ is more than $2^{(|\chi(X)|+2) \times f_{(|\chi(X)|+2)}(n)}$ or equal to $0$.
Then, for any sequence of parameterizing morphisms $(\mathcal{S})$ of $X\to \Spec K(Z)$, there exists a sequence of parameterizing morphisms $(\mathfrak{S}')$ of a proper hyperbolic polycurve $\mathfrak{X}'\to Z$ such that the base change of $(\mathfrak{X}', (\mathfrak{S}'))$ to $\Spec K(Z)$ is isomorphic to $(X,(\mathcal{S}))$ and $\mathfrak{X}$ is canonically isomorphic to $\mathfrak{X}'$.
\end{enumerate}
\label{uniqueness}
\end{thm}

\begin{proof}
Theorem \ref{uniqueness}.2 follows from Theorem \ref{uniqueness}.1 and the fact that 
$2g_{\mathcal{S}}\leq\left|\chi(X)\right|+2$ for any sequence of parametrizing morphisms $(\mathcal{S})$ of $X\to Z$.
The uniqueness portion of Theorem \ref{uniqueness}.1 follows from Theorem \ref{polypurity}.1.
We show the rest of Theorem \ref{uniqueness}.1.
Let
$$(\mathcal{S}): X=X_{n} \to \ldots \to X_{0} = \Spec K(Z)$$
be a sequence of parametrizing morphisms of $X$.
By Theorem \ref{polypurity}, we may assume that $Z$ is the spectrum of a discrete valuation ring $O_{K(Z)}$.
Then Theorem \ref{uniqueness}.1 follows from {\cite[Theorem 1.2.1, Theorem 1.2.3, and Theorem 1.3]{Nag}}.
\end{proof}

\begin{cor}
Let $Z, K(Z), \mathfrak{X},$ and $X$ be as in Theorem \ref{uniqueness}.
Suppose that $Z$ is a Dedekind scheme and that the residual characteristic of every closed point is more than $2^{(|\chi(X)|+2) \times f_{(|\chi(X)|+2)}(n)}$ or equal to $0$.
Then $X$ has good reduction with respect to any sequence of parameterizing morphisms of the proper hyperbolic polycurve $X \to \Spec K(Z)$ and $X$ has a proper N\'eron model.
\label{cor}
\end{cor}

\begin{proof}
Corollary \ref{cor} follows from Theorem \ref{uniqueness}, Proposition \ref{polyneron}, and {\cite[Corollary 2.5]{Liu}}.
\end{proof}

\label{neronsec}

\section{The Shafarevich conjecture for proper hyperbolic polycurves}
In this section, we prove the Shafarevich conjecture for proper hyperbolic polycurves.
Firstly, we will recall the Shafarevich conjecture for proper hyperbolic curves over finitely generated fields of characteristic of $0$, which was proved by Faltings. Then we will prove the main theorem (Theorem \ref{Shaf}) by using Faltings's result and results of Section \ref{neronsec}.

\begin{prop}[Faltings] 
Let $S$ be a normal connected scheme flat of finite type over $\Spec \Z$. Let $g \geq 2$ be an integer.
Then there exist at most finitely many isomorphism classes of proper hyperbolic curves of genus $g$ over $K(S)$ which have a smooth proper model over $S$.
\label{ShafFal}
\end{prop}
\begin{proof}
We may shrink $S$ so that $S$ is regular.
Now Proposition \ref{ShafFal} follows from {\cite[VI, \S1, Theorem 2]{Fal}} and the Torelli theorem.
\end{proof}

\begin{prop}
Let $S$ be a normal connected scheme flat of finite type over $\Spec \Z$. Let $\chi$ be an integer, and $n$ a positive integer. Then there exist at most finitely many isomorphism classes of proper hyperbolic polycurves of dimension $n$ with Euler-Poincar\'e characteristic $\chi$ over $S$ with a sequence of parameterizing morphisms. 
\label{finiteness}
\end{prop}

\begin{rem}
As written in Definition \ref{hyperbolicpoly}.3, the dimension of a proper hyperbolic polycurve is bounded by the absolute value of its Euler-Poincar\'{e} characteristic. Therefore, Proposition \ref{finiteness} is still true even if we do not fix the dimension.
\end{rem}

\begin{proof}
Let $A(n, \chi, S)$ be the set of isomorphism classes of proper hyperbolic polycurves with a sequence of parameterizing morphisms as in the statement of Proposition \ref{finiteness}.
We will prove Proposition \ref{finiteness} by induction on $n$.
The case of $n=1$ follows from Lemma \ref{moduliunique} and Proposition \ref{ShafFal}.
Let $n$ be a positive integer greater than $1$.
Take a pair $(\mathfrak{X},(\mathfrak{S})) \in A(n,\chi,S)$ with 
\begin{align*}
(\mathfrak{S}) \colon \mathfrak{X}=\mathfrak{X}_{n} \to \mathfrak{X}_{n-1} \to \cdots \to \mathfrak{X}_{1} \to \mathfrak{X}_{0} = S.
\end{align*}
Let $(\mathfrak{X}_{n-1}, (\mathfrak{S}'))$ be the proper hyperbolic polycurve with a sequence of parameterizing morphism cut out from $(\mathfrak{X}, \mathfrak{S})$. We have $(\mathfrak{X}_{n-1}, (\mathfrak{S}')) \in A(n-1, \chi(\mathfrak{X}_{n-1}),S)$ and $\left|\chi(\mathfrak{X}_{n-1})\right| \leq |\chi|$. 
By the induction hypothesis, we may fix the isomorphism class of $(\mathfrak{X}_{n-1}, (\mathfrak{S}'))$.
Since $\mathfrak{X}_{n-1}$ is a regular connected scheme flat of finite type over $\Spec \Z$, we have
$\mathfrak{X}_{n}\rightarrow \mathfrak{X}_{n-1} \in A(1,\chi(\mathfrak{X}_{n} \rightarrow \mathfrak{X}_{n-1}) ,\mathfrak{X}_{n-1})$.
Since $\left|\chi(\mathfrak{X}_{n}\to \mathfrak{X}_{n-1})\right| \leq \left| \chi \right|$, the desired finiteness follows from the case of $n=1$.
\end{proof}

\begin{thm}
Let $S$ be an integral scheme flat of finite type over $\Spec \Z$.
Let $\chi$ be an integer, and $n$ a positive integer.
Then there exist at most finitely many isomorphism classes of proper hyperbolic polycurves of dimension $n$ with Euler-Poincar\'e characteristic $\chi$ over $K(S)$ which have good reduction at any regular codimension $1$ point of $S$.
\label{Shaf}
\end{thm}
\begin{rem}
As in Proposition \ref{finiteness}, we do not need to fix the dimension.
\end{rem}

%

\begin{proof}
Note that we can replace $S$ by another nonempty open subscheme of $S$.
Therefore, we may assume that $S$ is regular and that the residual characteristic of any point of $S$ of codimension $1$ is more than $2^{(|\chi|+2) \times f_{(|\chi(X)|+2)}(n)}$ or equal to $0$.
For any proper hyperbolic polycurve $X$ of dimension $n$ with Euler-Poincar\'e characteristic $\chi$ over $K(S)$, one can equip $X$ with a sequence of parameterizing morphisms $(\mathcal{S})$ over $K(S)$.
By the assumption on the residual characteristics and Theorem \ref{uniqueness}, the pair $(X,(\mathcal{S}))$ extends to a proper hyperbolic polycurve with a sequence of parameterizing morphisms over $S$ uniquely. 
Therefore, it suffices to show the finiteness of the isomorphism classes of proper hyperbolic polycurves with a sequence of parameterizing morphisms of dimension $n$ with Euler-Poincar\'{e} characteristic $\chi$ over $S$.
This follows from Proposition \ref{finiteness}.
\end{proof}

\label{sectshaf}

\section{An application of the Shafarevich conjecture}
In this section, we show the finiteness of isomorphism classes of proper hyperbolic polycurves over a fixed number field satisfying a condition determined by their \'etale fundamental groups.
This finiteness was proved in \cite{Saw} (cf.\,Remark \ref{Sawada}) by examining the geometric \'etale fundamental groups of proper hyperbolic polycurves.
We show this by using the Shafarevich conjecture of proper hyperbolic polycurves (cf.\,Theorem \ref{Shaf}) and \cite[Theorem 1.3]{Nag}.

Let $L$ be a field and $X \to \Spec L$ a proper hyperbolic polycurve.
Take a geometric point $\ast$ of $X$ and write $\pi_{1}(X, \ast) \to G_{L}$ for the surjective homomorphism between the \'etale fundamental groups induced by $X \to \Spec L$.
Note that $G_{L}$ is isomorphic to the absolute Galois group of $L$ defined by $\ast$.

\begin{cor}
Let $K$ be a field finitely generated over $\Q$, $G_{K}$ its absolute Galois group, and $\Pi \to G_{K}$ a surjective homomorphism of profinite groups.
Then there are at most finitely many $K$-isomorphism classes of proper hyperbolic polycurves whose \'etale fundamental groups are  isomorphic to $\Pi$ over $G_{K}$.
\label{sawcor}
\end{cor}

\begin{rem}
Sawada proved Corollary \ref{sawcor} in the case where $K$ is a generalized sub-$p$-adic field.
Moreover, he treated general (not necessarily proper) hyperbolic polycurves.
We give another proof of Corollary \ref{sawcor} because we can prove Corollary \ref{sawcor} immediately by using Theorem \ref{Shaf} and \cite[Theorem 1.3]{Nag} under the assumptions of Corollary \ref{sawcor}.
\label{Sawada}
\end{rem}

\begin{proof}
If the \'etale fundamental group of a proper hyperbolic polycurve over $K$ is isomorphic to $\Pi$ over $G_{K}$, its Euler-Poincar\'e characteristic coincides with $\chi(\mathrm{Ker}(\Pi\to G_{K}),\F_{2})$ by Remark \ref{fundEuler}.
Therefore, by the last sentence of Definition \ref{hyperbolicpoly}.3, it suffices to show that there are at most finitely many $K$-isomorphism classes of proper hyperbolic polycurves of dimension $n$ whose \'etale fundamental groups are  isomorphic to $\Pi$ over $G_{K}$ for every natural number $n$ (cf.\,Remark \ref{dimension}).
Moreover, we may assume that there exists a proper hyperbolic polycurve $X_{K}$ of dimension $n$ over $K$ whose \'etale fundamental groups are  isomorphic to $\Pi$ over $G_{K}$.
Take a regular connected flat scheme $S$ of finite type over $\Spec \Z$ whose function field is isomorphic to $K$.
Since we can replace $S$ by its open dense subscheme, we may assume that there exists a proper hyperbolic polycurve $X\to S$ whose base change to $K$ is isomorphic to $X_{K}\to \Spec K$ by Theorem \ref{uniqueness}.2.
Moreover, we replace $S$ by its sufficiently small open dense subscheme so that we can apply \cite[Theorem 1.3]{Nag} in this situation for $l=2$.
Hence, another proper hyperbolic polycurve $X'$ of dimension $n$ over $K$ whose \'etale fundamental group is isomorphic to $\Pi$ over $G_{K}$, has good reduction at any point of $S$ of codimension 1.
By Theorem \ref{polypurity}.1, $X'$ extends to a proper hyperbolic polycurve over $S$.
Hence, by Theorem \ref{Shaf}, Corollary \ref{sawcor} holds.
\end{proof}

\begin{rem}
In fact, as in the proof of \cite{Saw}, the dimension of a (proper) hyperbolic polycurve $X$ over a field $L$ of characteristic $0$ is determined by the profinite group $\mathrm{Ker}(\pi_{1}(X, \ast) \to G_{L})$.
\label{dimension}
\end{rem}

\label{sectpresc}

\end{document}